\documentclass{article}

\usepackage{makeidx}
\usepackage{graphicx}
\usepackage{multicol}
\usepackage{amsmath,amssymb,amsfonts, amsthm}

\makeindex

\newcommand{\Z}{\mathbb Z}

\newcommand{\su}{\subseteq}

\newtheorem{thm}{Theorem}
\newtheorem{lem}{Lemma}
\newtheorem{cor}{Corollary}
\newtheorem{obs}{Observation}
\theoremstyle{definition}
\newtheorem{defn}{Definition}

\author{Santanu Mondal, Krishnendu Paul, Shameek Paul%
\thanks{E-mail addresses: \texttt{saanmondal008@gmail.com, krishnendu.paul30@gmail.com, shameek.paul@rkmvu.ac.in}}}

\date{}

\begin{document}
\baselineskip=14.5pt

\title {Extremal sequences for the unit-weighted Gao constant of $\Z_n$}

\maketitle

\begin{abstract}
For $A\su\Z_n$, the $A$-weighted Gao constant $E_A(n)$ is defined to be the smallest natural number $k$, such that any sequence of $k$ elements in $\Z_n$ has a subsequence of length $n$, whose $A$-weighted sum is zero. Sequences of length $E_A(n)-1$ in $\Z_n$, which do not have any $A$-weighted zero-sum subsequence of length $n$ are called $A$-extremal sequences for the Gao constant. Such a sequence which has $n-1$ zeroes is said to be of the standard type. When $A=U(n)$, is the set of units in $\Z_n$, where $n$ is odd, we characterize all such sequences and show that they are of the standard type. When $n$ is even, we give examples of such sequences which are not of the standard type. We also characterize the $U(n)$-extremal sequences for the Gao constant, when $n=2^rp$, where $p$ is an odd prime.
\end{abstract}

\bigskip


\vspace{.7cm}

\section{Introduction}\label{0}

\begin{defn}

Let $R$ be a ring and let $A\su R$. A subsequence $T$ of a sequence $S:(x_1,x_2,\ldots, x_{k})$ in $R$ is called an {\it $A$-weighted zero-sum subsequence} if the set $I:=\{i:x_i\in T\}$ is non-empty and $\forall~i\in I,~\exists~a_i\in A$ such that $\sum_{i\in I} a_i x_i = 0$. 
\end{defn}

\begin{defn} 
Given a ring $R$ and a subset $A\su R$, the $A$-weighted Davenport constant $D_A(R)$ is the least positive integer $k$ such that any sequence in $R$ of length $k$ has an $A$-weighted zero-sum subsequence. 
\end{defn}

\begin{defn}
Given a ring $R$ and a subset $A\su R$, the $A$-weighted Gao constant $E_A(R)$ is the least positive integer $k$ such that any sequence in $R$ of length $k$ has an $A$-weighted zero-sum subsequence of length $|R|$. 
\end{defn}

We denote the ring $\Z/n\Z$ by $\Z_n$. For a divisor $m$ of $n$, we define the natural map $\Z_n\to\Z_m$ to be the map which sends $x+n\Z\mapsto x+m\Z$. Let $U(n)$ denote the group of units in $\Z_n$.  

\smallskip

When $A\su \Z_n$, we denote the constants $D_A(\Z_n)$ and $E_A(\Z_n)$ by $D_A(n)$ and $E_A(n)$ respectively. From Theorem 1.2 of \cite{YZ}, we have $E_A(n)=D_A(n)+n-1$.         

\smallskip

\begin{defn}
Let $A\su \Z_n$. A sequence in $\Z_n$ of length $E_A(n)-1$ which does not have any $A$-weighted zero-sum subsequence of length $n$, is called an $A$-extremal sequence for the Gao constant. A sequence in $\Z_n$ of length $D_A(n)-1$ which does not have any $A$-weighted zero-sum subsequence, is called an $A$-extremal sequence for the Davenport constant. 
\end{defn}

\begin{defn}
Let $A$ be a subgroup of $U(n)$ and let $S:(x_1,\ldots,x_k)$ and $T:(y_1,\ldots,y_k)$ be sequences in $\Z_n$. We say that $S$ and $T$ are $A$-equivalent if there is a unit $c\in U(n)$, a permutation $\sigma\in S_k$ and we can find $a_1,\ldots,a_k\in A$ such that for $1\leq i\leq k$, we have $c\,y_{\sigma(i)}=a_ix_i$. 
\end{defn}

{\it Remark}: If $S$ is an $A$-extremal sequence for the Gao (resp. Davenport) constant and if $S$ and $T$ are $A$-equivalent, then $T$ is also an $A$-extremal sequence for the Gao (resp. Davenport) constant. 

\smallskip

From Theorem 1.3 of \cite{G} or from Theorem 1 of \cite{L},  $E_{U(n)}(n)=n+\Omega(n)$, for any $n$. So, if $S$ is a $U(n)$-extremal sequence for the Gao constant, then $S$ has length $n-1+\Omega(n)$. 

\smallskip

For $n$ odd, it was shown in Theorem 6 of \cite{AMP}, that a sequence in $\Z_n$ is a $U(n)$-extremal sequence for the Davenport constant if and only if it is $U(n)$-equivalent to a sequence of the following form:
$$(\,b_1,\,p_1b_2,\,p_1p_2b_3,\,\ldots ,\,p_1p_2\cdots p_{k-1}b_k\,)$$
where for $1\leq i\leq k$, $p_i$ is a prime and $b_i$ is coprime to $p_i$ and $n=p_1\ldots p_k$.

\smallskip

When $n=p^r$, where $p$ is an odd prime, in Theorem 3 of \cite{AMP}, it was shown that a sequence in $\Z_n$ is a $U(n)$-extremal sequence for the Gao constant if and only if it is $U(n)$-equivalent to the sequence 

$$(\,1,\,p,\,p^2,\,\ldots,\,p^{r-1},\,\overbrace{0,\,\ldots,\,0}^\text{$n-1$ times}\,)$$ 

In this article, we have proved the following results: 

\begin{itemize}

\item If $n$ is odd, a sequence in $\Z_n$ is a $U(n)$-extremal sequence for the Gao constant if and only if it is of the standard type, i.e., it has $n-1$ zeroes. 

\item For any even $n$, we give examples of $U(n)$-extremal sequences for the Gao constant which are not of the standard type. 

\item For $n=2^rp$, where $p$ is an odd prime, we characterize the $U(n)$-extremal sequences for the Gao constant.  
\end{itemize}

\section{When $n$ is any odd number}

For the next theorem, we need the following (\cite{G}, Lemma 2.1 (ii)), which we restate here using our terminology:

\begin{lem}\label{gri}
Let $A=U(p^r)$, where $p$ is an odd prime. If a sequence $S$ over $\Z_{p^r}$ has at least two terms coprime to $p$, then $S$ is an $A$-weighted zero-sum sequence. 
\end{lem}

The next result is Lemma 12 of \cite{SKS3}. 

\begin{lem}\label{lifts}
Let $S:(x_1,\ldots,x_k)$ be a sequence in $\Z_n$, let $d$ be a proper divisor of $n$ which divides every element of $S$ and let $n'=n/d$. For $1\leq i\leq k$, let $x_i'$ denote the image of $x_i/d$ under the natural map $f:\Z_n\to\Z_{n'}$. Let $S':(x_1',\ldots,x_k')$. Let $A\su\Z_n$ and let $A'\su f(A)$. Suppose $S'$ is an $A'$-weighted zero-sum sequence. Then $S$ is an $A$-weighted zero-sum sequence.  
\end{lem}

\begin{lem}\label{star}
Suppose $n=p_1\ldots p_k$, where the $p_i$s are primes. Let 
$$\hspace{1cm} S:(\,b_1,\,p_1b_2,\,p_1p_2b_3,\,\ldots ,\,p_1p_2\cdots p_{k-1}b_k,\,\overbrace{0,\,\ldots,\,0}^\text{$n-1$ times}\,)\hspace{2cm}{(\star)}$$
where for $1\leq i\leq k$, $b_i$ is coprime to $p_i$. Then, $S$ is a $U(n)$-extremal sequence for the Gao constant. 
\end{lem}

\begin{proof}
Let $T$ be a $U(n)$-weighted zero-sum subsequence of $S$ of length $n$. Suppose the first term of $T$ is the $i^{th}$ term of $S$, where $1\leq i\leq k$. Then, $p_1\ldots p_i$ divides all the other terms of $T$. So, we get the contradiction that $b_i$ is divisible by $p_i$. Thus, the first term of $T$ cannot be any of the first $k$ terms of $S$. As $T$ has length $n$, so, the first term of $T$ cannot be any of the remaining $n-1$ terms of $S$. Thus, $T$ cannot exist. 
\end{proof}

\begin{thm}\label{odde}
Let $n$ be odd. Then a sequence in $\Z_n$ is a $U(n)$-extremal sequence for the Gao constant if and only if it is $U(n)$-equivalent to a sequence having the form $(\star)$. 

\end{thm}

\begin{proof}
From Lemma \ref{star}, a sequence which is $U(n)$-equivalent to a sequence having the form $(\star)$ is a $U(n)$-extremal sequence for the Gao constant. 

\smallskip

Let $S:(x_1,\,\ldots,\,x_l)$ be a $U(n)$-extremal sequence for the Gao constant. Suppose for each prime divisor $p$ of $n$, at least two terms of $S$ are coprime to $p$. As $2\omega(n)<n$, so we can find a subsequence $T$ of $S$ of length $n$ such that for each prime divisor $p$ of $n$, at least two terms of $T$ are coprime to $p$. As $n$ is odd, so, by Lemma \ref{gri}, we get the contradiction that $T$ is a $U(n)$-weighted zero-sum subsequence of $S$ of length $n$. Thus, there is a prime divisor $p$ of $n$, such that at most one term of $S$ is coprime to $p$.  

\smallskip

Suppose $p$ divides all the terms of $S$. Let $n'=n/p$, $A'=U(n')$ and let $f:\Z_n\to\Z_{n'}$ be the natural map. Let $S':(x_1',\,\ldots,\,x_l')$ denote the sequence in $\Z_{n'}$ where $x_i'=f(x_i/p)$, for $1\leq i\leq l$. As $E_{A'}(n')=n'+\Omega(n')$ and as the length of $S'$ is $n-1+\Omega(n)=(p-1)n'+n'+\Omega(n')$, so, $S'$ contains $p$ disjoint subsequences $S_1',\ldots,S_p'$ which are $A'$-weighted zero-sum sequences of length $n'$. So, their union $T'$ is an $A'$-weighted zero-sum subsequence of $S'$ of length $n$. Thus, by Lemma \ref{lifts}, we get the contradiction that $S$ has an $A$-weighted zero-sum subsequence of length $n$. 

\smallskip

Thus, there is exactly one term of $S$ which is coprime to $p$, say $x_1$. By repeating the arguments in the above two paragraphs for the sequence $(x_2',\ldots,x_l')$, we see that there is a prime divisor $p'$ of $n'$ such that exactly one term of this sequence, say $x_2'$ is coprime to $p'$. Let us denote $p$ by $p_1$ and $p'$ by $p_2$. Let $k=\Omega(n)$. For $2\leq i\leq k$, we get prime divisors $p_i$ of $n$, such that $x_i/(p_1\ldots p_{i-1})$ is coprime to $p_i$ and $p_1p_2\ldots p_i$ divides $x_j$ for $j>i$. Thus, $x_j=0$ for all $j>k$. Hence, $S$ is $U(n)$-equivalent to a sequence having the form $(\star)$.  
\end{proof}

\begin{defn}
An $A$-extremal sequence for the Gao constant which has $n-1$ zeroes, is said to be of the standard type. 
\end{defn}

Remark: For $A\su\Z_n$, $E_A(n)=n-1+D_A(n)$. So, if $S$ is an $A$-extremal sequence for the Gao constant which is of the standard type, then the subsequence consisting of the non-zero terms of $S$ will be an $A$-extremal sequence for the Davenport constant. By Lemma \ref{star}, a sequence having the form $(\star)$ is a $U(n)$-extremal sequence for the Gao constant. Such a sequence is of the standard type. 

\smallskip

However, when $n$ is even, there are $U(n)$-extremal sequences for the Gao constant which are not of the standard type.

\section{When $n$ is any even number}

The next result is Theorem 10 of \cite{SKS5}.

\begin{thm}
Let $n=2^r$ and $r\geq 2$. A sequence in $\Z_n$ is a $U(n)$-extremal sequence for the Gao constant if and only if it is $U(n)$-equivalent to a sequence $S$ of length $n+r-1$, such that, for each $i$ with $0\leq i\leq r-2$, $2^i$ occurs exactly once as a term of $S$, $2^{r-1}$ occurs an odd number of times as a term of $S$ and the remaining terms of $S$ are zero. 
\end{thm}

For example, a sequence in $\Z_8$ is a $U(8)$-extremal sequence for the Gao constant if and only if it is $U(8)$-equivalent to one of the following sequences: 

\smallskip

$$(\,1,\,2,\,4,\,0,\,0,\,0,\,0,\,0,\,0,\,0\,),~(\,1,\,2,\,4,\,4,\,4,\,0,\,0,\,0,\,0,\,0\,)$$
$$(\,1,\,2,\,4,\,4,\,4,\,4,\,4,\,0,\,0,\,0\,),~(\,1,\,2,\,4,\,4,\,4,\,4,\,4,\,4,\,4,\,0\,)$$ 

\begin{lem}\label{ast}
Suppose $n=2p_1\ldots p_k$, where the $p_i$s are primes.

\smallskip 
If $k\geq 2$, let 
$$\hspace{.5cm}S:(\,b_1,\,p_1b_2,\,p_1p_2b_3,\,\ldots,\,p_1\cdots p_{k-1}b_k,\,\overbrace {n/2,\,\ldots,\,n/2}^\text{$m$ times},\,\overbrace{\,0,\,\ldots,\,0}^\text{$n-m$ times}\,)\hspace{1cm}{(\ast)}$$
and if $k=1$ and $p_1=p$, let 
$$S:(\,b_1,\,\overbrace {p,\,\ldots,\,p}^\text{$m$ times},\,\overbrace{\,0,\,\ldots,\,0}^\text{$n-m$ times}\,)$$
where $m$ is odd and for $1\leq i\leq k$, we have $b_i$ is coprime to $p_i$. Then, $S$ is a $U(n)$-extremal sequence for the Gao constant.
\end{lem}

\begin{proof}
Suppose $T$ is a $U(n)$-weighted zero-sum subsequence of $S$. Then, the first term of $T$ cannot be any of the first $k$ terms of $S$. We claim that $T$ cannot consist of the last $n$ terms of $S$.

\smallskip

Let $f:\Z_n\to\Z_2$ be the natural map and let $T'$ be the sequence whose terms are obtained by dividing the terms of $T$ by $n/2$ and then taking their images under $f$. The sequence $T'$ in $\Z_2$ has an odd number of non-zero terms which are all equal to one. So, we see that $T'$ cannot be a zero-sum sequence. As the map $f$ sends elements of $U(n)$ to 1, hence, $T$ cannot be a $U(n)$-weighted zero-sum sequence. 
\end{proof} 

We now give an example of a $U(n)$-extremal sequence for the Gao constant whose all terms may be non-zero, when $n$ is even.

\begin{lem}\label{2ast}
Suppose $n=2^{r+1}p_1\ldots p_k$, where the $p_i$s are odd primes. 

\smallskip

If $k\geq 2$, let $S$ be the sequence 
$$(\,a_0,\,2a_1,\,\ldots,\,2^{r-1}a_{r-1},\,c,\,2^rb_1,\,2^rp_1b_2,\,\ldots,\,2^rp_1\cdots p_{k-1}b_k,\,\overbrace {n/2,\,\ldots,\,n/2}^\text{$n-1$ times}\,)$$
and if $k=1$, let 
$$\hspace{2cm}S:(\,a_0,\,2a_1,\,\ldots,\,2^{r-1}a_{r-1},\,c,\,2^r,\,\overbrace {n/2,\,\ldots,\,n/2}^\text{$n-1$ times}\,)\hspace{1cm}{(\ast\ast)}$$ 
where for $0\leq i\leq r-1$, $a_i$ is odd, $c$ is divisible by $2^{r+1}$ and for $1\leq i\leq k$, we have $b_i$ is odd and coprime to $p_i$. Then, $S$ is a $U(n)$-extremal sequence for the Gao constant. 
\end{lem}

\begin{proof}
We give the proof in the case $k\geq 2$. 
Let $T$ be a $U(n)$-weighted zero-sum subsequence of $S$ of length $n$. Then, the first term of $T$ cannot be any of the first $r$ terms. Suppose the first term of $T$ is $c$. All the other terms of $T$ are of the form $2^ra$, where $a$ is odd. As $n$ is even, $T$ has an odd number (viz. $n-1$) of such terms. So, a $U(n)$-weighted sum of these terms is also of the same form and hence cannot be divisible by $2^{r+1}$. Thus, the first term of $T$ cannot be $c$. It is easy to see that the first term of $T$ cannot be any of the next $k$ terms of $S$. As there are only $n-1$ terms remaining, so, $S$ has no $U(n)$-weighted zero-sum subsequence of length $n$.
\end{proof}

\section{When $n=2p$, where $p$ is an odd prime}

We begin this section by quoting Observation 2.2 in \cite{G}. 

\begin{obs}\label{obs}
For every prime divisor $p$ of $n$, let $v_p(n)$ denote the highest power of $p$ which divides $n$. Let $S$ be a sequence in $\Z_n$. For each prime divisor $p$ of $n$, let $S^{(p)}$ denote the image of $S$ under the natural map $\Z_n\to\Z_{p^{v_p(n)}}$. Then, $S$ is a $U(n)$-weighted zero-sum sequence in $\Z_n\iff$ for every prime divisor $p$ of $n$, $S^{(p)}$ is a $U(p^s)$-weighted zero-sum sequence in $\Z_{p^s}$, where $s=v_p(n)$. 
\end{obs}

\begin{lem}\label{w}
 Let $n=2p$, where $p$ is an odd prime. If $W$ is a sequence in $\Z_n$ such that $W^{(2)}$ has an even number of ones and $W$ doesn't have exactly one term which is coprime to $p$. Then, $W$ is a $U(n)$-weighted zero-sum sequence.
\end{lem}

\begin{proof}
 By Lemma \ref{gri}, if $W^{(p)}$ has at least two units, then it is a $U(p)$-weighted zero-sum sequence. If it has no non-zero term, then again it is a $U(p)$-weighted zero-sum sequence. As $W^{(2)}$ has an even number of ones, so, it is a zero-sum sequence. Thus, by Observation \ref{obs}, $W$ is a $U(n)$-weighted zero-sum sequence. 
\end{proof}

\begin{thm}\label{2p}
Let $n=2p$, where $p$ is an odd prime. Then, $S$ is a $U(n)$-extremal sequence for the Gao constant if and only if $S$ is $U(n)$-equivalent to a sequence of length $n+1$ having the form $(\star)$, $(\ast)$ or $(\ast\ast)$. 
\end{thm}

\begin{proof}
From Lemmas \ref{star}, \ref{ast} and \ref{2ast}, we see that, when $n$ is even, a sequence in $\Z_n$ which is $U(n)$-equivalent to a sequence having the form $(\star)$, $(\ast)$ or $(\ast\ast)$, is a $U(n)$-extremal sequence for the Gao constant. 

\smallskip

Let $n=2p$, where $p$ is an odd prime. Suppose $S$ is a $U(n)$-extremal sequence for the Gao constant. As $E_{U(n)}(n)=n+\Omega(n)$, so, $S$ has length $n+1$. 

\smallskip

Suppose all the terms of $S^{(2)}$ are zero. If $S^{(p)}$ has at most one non-zero term, then we get the contradiction that $S$ has a subsequence of length $n$ whose all terms are zero. So, $S^{(p)}$ must have at least two non-zero terms. Let $T$ denote a subsequence of $S$ of length $n$ such that $T^{(p)}$ has at least two non-zero terms. Then, by Lemma \ref{w}, we get the contradiction that $T$ is a $U(n)$-weighted zero-sum subsequence of $S$ of length $n$. So all terms of $S^{(2)}$ cannot be zero. By a similar argument, all terms of $S^{(2)}$ cannot be one. 

\smallskip

Suppose $S^{(2)}$ has exactly one non-zero term. Let $T$ denote the subsequence of $S$ of length $n$ such that all terms of $T^{(2)}$ are zero. If $T^{(p)}$ has at least two units, then, by Lemma \ref{w}, we get the contradiction that $T$ is a $U(n)$-weighted zero-sum subsequence of $S$ of length $n$. If $T^{(p)}$ has no units, then all terms of $T^{(p)}$ are zero and so, all terms of $T$ are zero. So, we see that $T^{(p)}$ has exactly one unit. Thus, $S$ is $U(n)$-equivalent to the sequence $(a,2b,0,\ldots,0)$, where there are $n-1$ zeroes, $a$ is odd and $b$ is coprime to $p$. 

\smallskip 
By a similar argument, if $S^{(2)}$ has exactly one term which is zero and if $T$ denotes the subsequence of $S$ of length $n$ such that all the terms of $T^{(2)}$ are one, then $T^{(p)}$ must have exactly one unit. Thus, $S$ is $U(n)$-equivalent to the sequence $(1,c,p,\ldots,p)$, where $p$ occurs  $n-1$ times and $c$ is even. 

\smallskip

Suppose $S^{(2)}$ has at least two terms which are zero and at least two terms which are 1. Suppose $S^{(p)}$ does not have exactly one unit. Then, by using Lemma \ref{w}, we will get a $U(n)$-weighted zero-sum subsequence of $S$ of length $n$. So, $S^{(p)}$ has exactly one unit. Let $T$ be the subsequence of $S$ of length $n$, such that $T^{(p)}$ is the zero sequence. 

\smallskip 

If $T^{(2)}$ has an even number of ones, then, $T^{(2)}$ will be a zero-sum sequence and so, we get the contradiction that $T$ is a zero-sum subsequence of $S$ of length $n$. Thus, $T^{(2)}$ has an odd number of ones and hence $S$ is $U(n)$-equivalent to the sequence $(b,p,\ldots,p,0,\ldots,0)$, where, $b$ is coprime to $p$ and $p$ occurs an odd number of times. 

\smallskip

Hence, when $n=2p$, where $p$ is an odd prime, a sequence in $\Z_n$ is a $U(n)$-extremal sequence for the Gao constant if and only if it is $U(n)$-equivalent to a sequence of length $n+1$  having one of the following forms:

\begin{itemize}
  \item $(a,\,2b,\,0,\,\ldots,\,0)$, where $a$ is odd and $b$ is coprime to $p$. \\
  This has the form $(\star)$. 
  
 \item $(b,\,\overbrace{p,\,\ldots,\,p}^\text{$m$ times},\,0,\,\ldots,\,0)$, where $b$ is coprime to $p$ and $m$ is odd. \\
 This has the form $(\ast)$. 
 
 \item $(1,\,c,\,p,\,\ldots,\,p)$, where $c$ is even. This has the form $(\ast\ast)$. 
 \qedhere
\end{itemize}
\end{proof}

\section{When $n=2^rp$, where $p$ is an odd prime}

The next result is Lemma 1 (ii) of \cite{L}. 

\begin{lem}\label{2r0}
Let $n=2^r$. If a sequence $S$ in $\Z_n$ has an even number (at least two) of units, then it is a $U(n)$-weighted zero-sum sequence.  
\end{lem}

\begin{lem}\label{12}
Let $n=2^rp$, where $p$ is an odd prime. Let $S$ be a sequence in $\Z_n$ of length at least $n+2$ such that $S$ does not have any $U(n)$-weighted zero-sum subsequence of length $n$. Then, $S$ has at most two odd terms. 

\smallskip

Also, if $S$ has exactly two odd terms, then exactly one of the odd terms is a unit and all the other terms of $S$ are divisible by $p$. 
\end{lem}

\begin{proof}
Let $S:(x_1,\ldots,x_k)$ be a sequence in $\Z_n$ of length at least $n+2$ such that $S$ does not have any $U(n)$-weighted zero-sum subsequence of length $n$. Suppose $S$ has at least three odd terms. We consider the following two possibilities. 

\smallskip 

{\it Case: Suppose $S$ has at least two terms which are coprime to $p$.}

\smallskip 

As the length of $S$ is at least $n+2$, so, (by removing a suitable $x_i$ if there are an odd number of odd terms) we can get a subsequence $T$ of $S$ of length $n$, which has an even number of odd terms and such that $T$ has at least two terms which are coprime to $p$. By Lemma \ref{2r0}, $T^{(2)}$ is a $U(2^r)$-weighted zero-sum sequence and by Lemma \ref{gri}, $T^{(p)}$ is a $U(p)$-weighted zero-sum sequence. So, by Observation \ref{obs}, we get the contradiction that $T$ is a $U(n)$-weighted zero-sum subsequence of $S$ having length $n$.

\smallskip

{\it Case: Suppose $S$ has at most one term which is coprime to $p$.}

\smallskip 

As the length of $S$ is at least $n+2$, so, (by removing at most two $x_i$s if needed) we can get a subsequence $T$ of $S$ of length $n$, which has an even number of odd terms and such that $T$ has no term which is coprime to $p$. By Lemma \ref{2r0}, $T^{(2)}$ is a $U(2^r)$-weighted zero-sum sequence. Also $T^{(p)}$ is a $U(p)$-weighted zero-sum sequence, as all its terms are zero. So, by Observation \ref{obs}, we get the contradiction that $T$ is a $U(n)$-weighted zero-sum subsequence of $S$ having length $n$. Thus, $S$ has at most two odd terms. 

\smallskip

Suppose $S$ has exactly two odd terms. By using a similar argument as in the second paragraph, we see that $S^{(p)}$ cannot have at least two units. By using a similar argument as in the third paragraph, we see that $S^{(p)}$ must have at least one unit. Thus, $S^{(p)}$ has exactly one unit and so, by a similar argument as in the third paragraph, both the odd terms cannot be divisible by $p$. Hence, exactly one of the two odd terms must be coprime to $p$, and so, it will be a unit. 
\end{proof}

\begin{thm}\label{exto}
Let $n=2^rp$, where $p$ is an odd prime and $r\geq 2$. Let $S$ be a sequence in $\Z_n$. Suppose $S$ is a $U(n)$-extremal sequence for the Gao constant. Then, one of the following two cases will occur: 

\smallskip

An odd multiple of $2^i$ occurs exactly once in $S$  for each $i$, such that $0\leq i\leq r-2$. 

\smallskip

There is a unique $j$ such that $0\leq j\leq r-2$ and an odd multiple of $2^j$ occurs exactly two times in $S$. Also, for any $i\neq j$ such that $0\leq i\leq r-2$, an odd multiple of $2^i$ occurs exactly once in $S$. 
\end{thm}

\begin{proof}
Let $S$ be a $U(n)$-extremal sequence for the Gao constant. As $E_{U(n)}(n)=n+\Omega(n)$, so, $S$ has length $k=n+r$. Suppose there exists $i$ such that $0\leq i\leq r-2$ and an odd multiple of $2^i$ occurs at least two times in $S$. 

\smallskip

Let $j$ be the smallest value of $i$, such that $0\leq i\leq r-2$ and an odd multiple of $2^i$ occurs at least two times in $S$. Suppose $S$ has at least three terms which are an odd multiple of $2^j$. As $S$ has at most one term which is an odd multiple of $2^i$, where $0\leq i\leq j-1$, so at most $j$ terms of $S$ will not be divisible by $2^j$. As $j\leq r-2$, so at least $k-(r-2)\geq n+r-(r-2)=n+2$ terms of $S$ are divisible by $2^j$.

\smallskip 

Let $U$ be the subsequence of $S$ consisting of all the terms of $S$ which are divisible by $2^j$. If $U$ has a $U(n)$-weighted zero-sum subsequence of length $n$, then $S$ will also have such a subsequence and this is not possible. As $U$ has length at least $n+2$, so, by Lemma \ref{12}, at most two terms of $U$ are an odd multiple of $2^j$. This contradicts the assumption that  $S$ has at least three terms which are an odd multiple of $2^j$. Thus, exactly two terms of $S$ are an odd multiple of $2^j$.  

\smallskip

By Lemma \ref{12}, exactly one of the two terms of $U$ which are an odd multiple of $2^j$, is coprime to $p$ and all the other terms of $U$ are divisible by $p$. Suppose $j\leq r-3$ and at least two terms of $S$ are odd multiples of $2^i$ for some $i$, where $j+1\leq i\leq r-2$. Let us assume that $i'$ is the smallest such value of $i$. Let $T$ denote the subsequence of $S$ consisting of all the terms which are divisible by $2^{i'}$. Then, $T$ has length at least $k-(i'+1)$. As $i'\leq r-2$, so $k-(i'+1)\geq n+r-(r-1)=n+1$.

\smallskip 

By using Lemma \ref{2r0}, we can find a $U(2^r)$-weighted zero-sum subsequence of length $n$ of $T^{(2)}$. Also, $T^{(p)}$ is the zero sequence, as all the terms of $U$ which are even multiples of $2^j$, are divisible by $p$, and so, all the terms of $T$ are divisible by $p$. Thus, by Observation \ref{obs}, we get the contradiction that $T$ (and hence $S$) has a $U(n)$-weighted zero-sum subsequence of length $n$.  Hence, if $j\leq r-3$, at most one term of $S$ is an odd multiple of $2^i$, for any $j+1\leq i\leq r-2$. 

\smallskip

Suppose for some $i$, such that $0\leq i\leq r-2$ and $i\neq j$, there is no term of $S$ which is an odd multiple of $2^i$. As any non-zero term of $\Z_{2^r}$ is a unit multiple of a power of 2 which is between $0$ and $r-1$, so, $S^{(2)}$ will have at least $k-(r-1)=n+1$ terms which are either zero or a unit multiple of $2^{r-1}$. So, we can find a subsequence $V$ of $S$ having length $n$, such that $V^{(2)}$ has an even number of terms which are a unit multiple of $2^{r-1}$. 

\smallskip 

So, $V^{(2)}$ is a zero-sum sequence in $\Z_{2^r}$. As all the even terms of $U$ are divisible by $p$ and as $V$ is a subsequence of $U$, so, $V^{(p)}$ is the  zero sequence. Thus, by Observation \ref{obs}, we get the contradiction that $V$ is a $U(n)$-weighted zero-sum subsequence of $S$ having length $n$. Hence, if $0\leq i\leq r-2$ and $i\neq j$, then there is at least one term (and hence  exactly one term) of $S$ which is an odd multiple of $2^i$. 
\end{proof}

\begin{cor}
Let $n=2^rp$, where $p$ is an odd prime and $r\geq 2$. Let $S$ be a sequence in $\Z_n$. Suppose $S$ is a $U(n)$-extremal sequence for the Gao constant. Then, $S$ has at least one and at most two odd terms.
\end{cor}

\begin{proof}
Let $S$ be a $U(n)$-extremal sequence for the Gao constant. As $E_{U(n)}(n)=n+\Omega(n)$, so, $S$ has length $k=n+r$. Then, $k\geq n+2$ as $r\geq 2$. So, by Lemma \ref{12}, $S$ has at most two odd terms. By Theorem \ref{exto}, $S$ has at least one odd term.
\end{proof}

\begin{thm}
Let $n=2^rp$, where $p$ is an odd prime and $r\geq 2$. Let $S$ be a sequence in $\Z_n$ which has exactly two odd terms. Then, $S$ is a $U(n)$-extremal sequence for the Gao constant if and only if $S$ is $U(n)$-equivalent to a sequence having the form $(\ast)$. 
\end{thm}

\begin{proof}
From Lemma \ref{ast}, we see that if a sequence $S$ is $U(n)$-equivalent to a sequence having the form $(\ast)$, then $S$ is a $U(n)$-extremal sequence for the Gao constant. 

\smallskip

Let $S$ be a sequence in $\Z_n$ which is a $U(n)$-extremal sequence for the Gao constant, where $n=2^rp$, where $p$ is an odd prime and $r\geq 2$. Then, $S$ has length $k=n+r$. If $S$ has exactly two odd terms, then, by Lemma \ref{12}, exactly one of the odd terms, is a unit and all the other terms of $S$ are divisible by $p$. By Theorem \ref{exto}, for each $i$ such that $1\leq i\leq r-2$, an odd multiple of $2^i$ occurs exactly once in $S$. 

\smallskip 

Thus, there are $k-r=n$ terms of $S$ which are either zero or an odd multiple of $2^{r-1}$. Let $T$ denote the subsequence of $S$ consisting of these $n$ terms. As the terms of $T$ are divisible by $p$, so, the  non-zero terms of $T$ must be equal to $n/2$. As $-1$ is a unit and as $T$ cannot be a $U(n)$-weighted zero-sum sequence, so, the number of terms of $T$ which are equal to $n/2$ must be odd. 

\smallskip 

Hence, $S$ is $U(n)$-equivalent to the sequence 
$$(\,1,\,p\,a_0,\,2p\,a_1,\,2^2p\,a_2,\,\ldots,\,2^{r-2}p\,a_{r-2},\,\overbrace{n/2,\,\ldots,\,n/2}^\text{$m$ times},\,0,\,\ldots,\,0\,)$$
where $m$ is odd and for $0\leq i\leq r-2$, $a_i$ is odd. This has the form $(\ast)$.   
\end{proof}

\begin{thm}
Let $n=2^rp$, where $p$ is an odd prime and $r\geq 2$. Let $S$ be a sequence in $\Z_n$ which has exactly one odd term. Then, $S$ is a $U(n)$-extremal sequence for the Gao constant if and only if $S$ is $U(n)$-equivalent to a sequence having the form $(\star)$, $(\ast)$ or $(\ast\ast)$. 
\end{thm}

\begin{proof}
From Lemmas \ref{star}, \ref{ast} and \ref{2ast}, we see that a sequence which is $U(n)$-equivalent to a sequence having the form $(\star)$, $(\ast)$ or $(\ast\ast)$, is a $U(n)$-extremal sequence for the Gao constant. 

\smallskip

Let $S$ be a $U(n)$-extremal sequence for the Gao constant, where $n=2^rp$, such that $p$ is an odd prime and $r\geq 2$. Then, $S$ has length $k=n+r$. By Theorem \ref{exto}, two cases can occur. 

\smallskip

{\it Case:} An odd multiple of $2^i$ occurs exactly once in $S$, for each $i$, such that $0\leq i\leq r-2$. 

\smallskip

Let $S_*$ denote the subsequence of $S$ consisting of the remaining $k-(r-1)=n+1$ terms. Then, every non-zero term of $S_*$ is an odd multiple of $2^{r-1}$. The arguments for the remaining part of this case of the proof, are similar to the arguments in Theorem \ref{2p}. So, we do not repeat those arguments and just state the conclusions directly. We see that  $S_*^{(2)}$ cannot have all terms zero or all terms non-zero. We use the fact that the sum of an even number of terms, all of which are odd multiples of $2^{r-1}$, will be zero in $\Z_{2^r}$. 

\smallskip 

If $S_*^{(2)}$ has exactly one non-zero term, then $S_*$ is $U(n)$-equivalent to the sequence $(2^{r-1}a,\,2^rb,\,0,\,\ldots,\,0)$, where there are $n-1$ zeroes, $a$ is odd and $b$ is coprime to $p$. If $S_*^{(2)}$ has exactly one term which is zero, then $S_*$ is $U(n)$-equivalent to the sequence $(c,\,2^{r-1},\,2^{r-1}p\,a_2,\,\ldots,\,2^{r-1}p\,a_n)$, where $c$ is divisible by $2^r$ and $a_i$ is odd, for $2\leq i\leq n$. If $S_*^{(2)}$ has at least two terms which are zero and at least two terms which are non-zero, then $S_*$ is $U(n)$-equivalent to the sequence $(2^{r-1}b,\,2^{r-1}p\,a_1,\,\ldots,\,2^{r-1}p\,a_m,\,0,\,\ldots,\,0)$, where $b$ is coprime to $p$, $m$ is odd and $a_i$ are odd, for $1\leq i\leq m$. 

\smallskip 

Thus, in the case when an odd multiple of $2^i$ occurs exactly once in $S$, for each $i$, such that $0\leq i\leq r-2$, we see that $S$ is $U(n)$-equivalent to a sequence having one of the following forms: 

\smallskip

$(a_0,\,2a_1,\,\ldots,\,2^{r-2}a_{r-2},\,2^{r-1}a_{r-1},\,2^rb,\overbrace{\,0,\,\ldots,\,0}^\text{$n-1$ times})$, where $a_i$ is odd, for $0\leq i\leq r-1$ and $b$ is coprime to $p$. This has the form $(\star)$.

\smallskip

$(a_0,\,2a_1,\,\ldots,\,2^{r-2}a_{r-2},\,2^{r-1}b,\,\overbrace{n/2,\,\ldots,\,n/2}^\text{$m$ times},\overbrace{0,\,\ldots,\,0}^\text{$n-m$ times})$, where $a_i$ is odd, for $0\leq i\leq r-2$, $b$ is coprime to $p$ and $m$ is odd. This has the form $(\ast)$.

\smallskip

$(a_0,\,2a_1,\,\ldots,\,2^{r-2}a_{r-2},\,c,\,2^{r-1},\,\overbrace{n/2,\,\ldots,\,n/2}^\text{$n-1$ times})$, where $a_i$ is odd, for $0\leq i\leq r-2$ and $c$ is divisible by $2^r$. This has the form $(\ast\ast)$. 

\bigskip

{\it Case:} There is a unique $j$ such that $1\leq j\leq r-2$ and an odd multiple of $2^j$ occurs exactly two times in $S$. Also, for any $i\neq j$ such that $0\leq i\leq r-2$, an odd multiple of $2^i$ occurs exactly once in $S$. 

\smallskip
 
Let $S_*$ denote the subsequence of $S$ consisting of the remaining $k-r=n$ terms. We claim that all the terms of $S_*$ are divisible by $p$.  
Consider the subsequence $T$ of $S$ consisting of all the terms of $S$ which are divisible by $2^j$. Then $T$ has length $k-j=n+r-j\geq n+2$. By Lemma \ref{12}, exactly one of the two terms which are odd multiples of $2^j$, is coprime to $p$ and all the other terms of $T$ are divisible by $p$. This proves our claim. Then, every non-zero term of $S_*$ is an odd multiple of $2^{r-1}$. If the number of non-zero terms in $S_*$ is  even, then $S_*$ will be a zero-sum subsequence of $S$ of length $n$. So, the number of non-zero terms in $S_*$ must be odd. Thus, $S$ is $U(n)$-equivalent to 
$$(a_0,\,2a_1,\,\ldots,\,2^{j-1}a_{j-1},\,2^j,\,2^jp\,a_j,\,\ldots,\,2^{r-2}p\,a_{r-2},\,\overbrace{n/2,\,\ldots,\,n/2}^\text{$m$ times},\overbrace{0,\,\ldots,\,0}^\text{$n-m$ times}),$$ 
where $a_i$ is odd, for $0\leq i\leq r-2$ and $m$ is odd. This has the form $(\ast)$. 
\end{proof}


\section{Concluding remarks}

From Theorem \ref{odde}, we see that when $n$ is odd, there is only one type of $U(n)$-extremal sequence for the Gao constant (and hence also for the Davenport  constant). However, in this paper, we have seen that for any even $n$, there are at least three types of $U(n)$-extremal sequences for the Gao constant. Sequences of the type $(\star)$ which have a term equal to $n/2$, are also of the type $(\ast)$. 

\smallskip 

When $n$ is even and $\omega(n)\leq 2$, it can be shown that there is only one type of $U(n)$-extremal sequence for the Davenport constant. When $n$ is even, with $\omega(n)\geq 3$, there are other types of $U(n)$-extremal sequences for the Davenport constant. So, we will get other types of $U(n)$-extremal sequences for the Gao constant (which are of the standard type).

\begin{lem}
Let $n=2p_1p_2\ldots p_r$ be squarefree where the $p_i$s are primes. Then, the following sequence is a $U(n)$-extremal sequence for the Davenport constant; $S:(a,\,\hat p_1,\,\hat p_2,\,\ldots,\,\hat p_r)$, where $\hat p_i=n/(2p_i)$, for $1\leq i\leq r$ and $a$ is chosen to be either 1 or 2, so that $S$ has an odd number of odd terms. 
\end{lem}

\begin{proof}
Let $n$ and $S$ be as in the statement of the lemma. Suppose $T$ is a $U(n)$-weighted zero-sum subsequence of $S$. Then, the first term of $T$ cannot be $\hat p_i$, for any $i$ with  $1\leq i\leq r$, as all the other terms of $T$ are divisible by $p_i$, if $i\leq r-1$. Suppose the first term of $T$ is $a$. If $T$ does not contain $\hat p_i$, for some $i$ with $1\leq i\leq r$, then, we get the contradiction that $a$ is divisible by $p_i$. Finally, as $S$ has an odd number of odd terms, so, any unit-weighted sum of these terms cannot be even. Thus, $T$ cannot be $S$. 
\end{proof}

The constant $C_A(n)$ has been defined, and the $U(n)$-extremal sequences for that constant have been characterized in \cite{SKS2}, when $n$ is odd. When $n$ is a power of 2, such sequences have been characterized in \cite{SKS5}. It will be interesting to characterize such sequences, for other even values of $n$.

\bigskip

{\bf Acknowledgement.}
Santanu Mondal would like to acknowledge CSIR, Govt. of 
India, for a research fellowship.


\begin{thebibliography}{12}
\bibitem{AMP}
S. D. Adhikari, I. Molla, S. Paul, {\em Extremal Sequences for Some Weighted Zero-Sum Constants for Cyclic Groups}, Combinatorial and Additive Number Theory IV, Springer Proc. in Math. \& Stat., {\bf 347}, (2021), 1-10, \\ 
https://doi.org/10.1007/978-3-030-67996-5\_1

\bibitem{G}
S. Griffiths, {\em The Erd\H{o}s-Ginzberg-Ziv theorem with units},
Discrete Mathematics, {\bf 308}, No. 23, (2008), 5473-5484, \\
https://doi.org/10.1016\%2Fj.disc.2007.09.060

\bibitem{L}
F. Luca, {\em  A generalization of a classical zero-sum problem},
Discrete Math. 307, No. 13, (2007), 1672--1678.

\bibitem{SKS2}
S. Mondal, K. Paul, S. Paul, {\em Extremal sequences for a particular weighted zero-sum constant}, submitted and e-print is available at \\
https://arxiv.org/abs/2111.01018 

\bibitem{SKS3}
S. Mondal, K. Paul, S. Paul, {\em Extremal sequences for the Davenport constant 
related to the Jacobi symbol}, submitted and e-print is available at \\
https://arxiv.org/abs/2111.14477 

\bibitem{SKS5}
S. Mondal, K. Paul, S. Paul, {\em On unit weighted zero-sum constants of $\Z_n$}, submitted and e-print is available at \\
https://arxiv.org/abs/2203.02665



\bibitem{YZ}
P. Yuan, X. Zeng,  Davenport constant with weights, {\it European
Journal of Combinatorics}, {\bf 31} (2010), 677-680.
\end{thebibliography}
\end{document}